\documentclass[a4paper,11pt,oneside,reqno]{amsart}
\usepackage{amscd,amsfonts,amssymb,amsmath,amsthm,bbm}
\usepackage{graphicx,psfrag}
\usepackage[english]{babel}
\usepackage[headinclude,DIV13]{typearea}

\areaset{15.1cm}{25.0cm}
 \usepackage[hypertex,
 colorlinks=true,
 linkcolor=blue,
 citecolor=blue,
 ]{hyperref}
\usepackage[T1]{fontenc}
\usepackage{mathrsfs}
\usepackage{amsmath,amsthm,amssymb}
\usepackage{latexsym}

\usepackage{enumerate}
\usepackage{mathrsfs}
\usepackage{stmaryrd}
\usepackage{amsopn}
\usepackage{amsmath}
\usepackage{amssymb}
\usepackage{amsfonts}
\usepackage{amsbsy}
\usepackage{amscd,indentfirst,epsfig}
\usepackage{hyperref}
\usepackage{amsfonts,amsmath,latexsym,amssymb,verbatim,amsbsy}
\usepackage{amsthm,psfrag}
\usepackage{dsfont}
\usepackage{colordvi}
\usepackage{pstricks}

\usepackage{bbding}
\usepackage{pifont}
\usepackage{enumerate}
\usepackage{relsize}

\renewcommand{\theequation}{\arabic{section}.\arabic{equation}}

\newtheorem{theo}{Theorem}[section]
\newtheorem{lemme}[theo]{Lemma}
\newtheorem{lemmeA}{Lemma A.}
\newtheorem{propoA}{Proposition A.}
\newtheorem{nbA}{Remark A.}

\newtheorem{propo}[theo]{Proposition}

\newtheorem{hyp}[theo]{Assumptions}
\newtheorem{nb}[theo]{Remark}

\newtheorem{exa}[theo]{Example}
\theoremstyle{definition}

\def \leq {\leqslant}
\def \geq {\geqslant}
\numberwithin{equation}{section}
\def\ind#1{\lower5pt\hbox{$\scriptstyle #1$}}
\def \d {\, \mathrm{d} }

\def \H {\mathcal{H}}

\def \bH {\mathbf{H}}
\def \e {\mathrm{e}}

\def\Q {\mathcal{Q}}

\def\R{{\mathbb R}}
\def \S {{\mathbb S}^2}
\def \E {\mathcal{E}}
\def \n {\widehat{n}}

\def \v {{v}}

\def \vb {\v_{\star}}

\newcommand{\RN}{{\mathbb{R}^n}}
\def \IR {\int_{\R^3}}
\def \IRR {\int_{\R^3 \times \R^3}}

\title[]
{{Two proofs of Haff's law for dissipative gases: the use of entropy and the weakly inelastic regime}}

\author{Ricardo J. Alonso \& Bertrand Lods}

\address{\textbf{Ricardo J. Alonso}, Department of Computational and Applied Mathematics, Rice University, Houston, TX 77005-1892.}
\email{ralonso@math.utexas.edu}

\address{\textbf{Bertrand Lods},  Dipartimento di Statistica e Matematica Applicata, Collegio Carlo Alberto, Universit\`{a} degli
Studi di Torino,  Corso Unione Sovietica, 218/bis, 10134 Torino, Italy.}\email{bertrand.lods@math.univ-bpclermont.fr}

\thanks{R. Alonso acknowledges the support from NSF grant DMS-0439872 and ONR grant N000140910290.}

\begin{document}

\maketitle

\begin{abstract} We revisit our recent contribution \cite{AloLo1} and give two simpler proofs of the so-called Haff's law for granular gases (with non-necessarily constant restitution coefficient). The first proof is based upon the use of entropy and asserts that Haff's law holds whenever the initial datum is of finite entropy. The second proof uses only the moments of the solutions and holds in some weakly inelasticity regime which has to be clearly defined whenever the restitution coefficient is non-constant.
\end{abstract}
 \medskip
%\tableofcontents
\section{Introduction}
\label{intro}
\setcounter{equation}{0}
The main objective of the present paper is to revisit our recent contribution \cite{AloLo1} and give two simpler proofs of the so-called Haff's law for granular gases (with non-necessarily constant restitution coefficient). The first proof is based upon the use of Boltzmann's entropy and asserts that Haff's law holds whenever the initial datum is of finite entropy. The second proof uses only the moments of the solutions and shows that Haff's law holds in some weakly inelasticity regime (see Theorem \ref{main} for a precise definition in the case of non-constant restitution coefficient) for initial datum with finite energy.
%%%%%%%%%%%%%%%%%%%%%
%%%%%%%%%%%%%%%%%%%%%
\subsection{Motivation} We consider in this paper freely cooling granular gases governed by the spatially homogeneous Boltzmann equation
\begin{equation}\label{cauch}\begin{cases}
\partial_t f(t,v)&=\Q_{e}(f,f)(t,v) \qquad \qquad t >0, \; v \in \mathbb{R}^3\\
f(0,v)&=f_0(v), \qquad \qquad  v \in \mathbb{R}^3
\end{cases}
\end{equation}
where the initial datum $f_0$ is a nonnegative velocity distribution such that
\begin{equation}\label{initial}\IR f_0(v)\d v=1, \quad \IR f_0(v) v \d v =0 \quad \text{ and } \quad \IR f_0(v)|v|^2 \d v < \infty.\end{equation}
The operator $\Q_e(f,f)$ is the
inelastic Boltzmann collision operator, expressing the effect of
binary collisions of particles. We assume here that the granular particles are perfectly smooth hard-spheres of mass
$m=1$. The
inelasticity of the collision mechanism is characterized by a single scalar
parameter known as the coefficient of normal restitution \(0 \leq
e \leq 1\).  Indeed, if $v$ and $\vb$  denote the velocities of two particles before they collide, their respective velocities $v'$ and $\vb'$ after collisions are such that
\begin{equation}\label{coef}
(u'\cdot \n)=-(u\cdot \n) \,e,
\end{equation}
where $\n \in \mathbb{S}^2$  determines the impact direction, i.e. $\n$ stands for the unit vector that points from the $v$-particle center to the $\vb$-particle center at the instant of impact.  Here above
$$u=v-\vb \quad \text{and}\quad u'=v'-\vb',$$
denote respectively the relative velocity before and after collision.  In this work, the restitution coefficient $e$ is assumed to be a function of the impact velocity, i.e.
$$e:=e(|u \cdot \n|).$$
In virtue of \eqref{coef} and the conservation of momentum, the post-collision velocities $(v',\vb')$ are given by
\begin{equation}
\label{transfpre}
  v'=v-\frac{1+e}{2}\,(u\cdot \n)\n,
\qquad \vb'=\vb+\frac{1+e}{2}\,(u\cdot \n)\n.
\end{equation} The main assumptions on the function $e(\cdot)$ are listed here after (see Proposition \ref{HYP1}) and ensure that the Jacobian of the above transformation \eqref{transfpre} is given by
$$J_e(|u\cdot \n|)=\left|\dfrac{\partial v'\partial \vb'}{\partial v \partial \vb}\right|=e(|u\cdot \n|)+ |u\cdot \n|e'(|u\cdot \n|)=:\vartheta'_e(|u\cdot \n|)$$
where $e'(\cdot)$ and $\vartheta_e'(\cdot)$ denote the derivative of $r \mapsto e(r)$ and $r \mapsto \vartheta_e(r)$ respectively (this prime symbol should not be confused with the one  we have chosen for the post-collisional velocity). We refer the reader to \cite{AlonsoIumj} for more details.  The main examples of restitution coefficient we shall deal within this paper are the following:
\begin{enumerate}[(1)\;]
\item The first fundamental example is the one of a \textit{\textbf{constant restitution coefficient}} for which $e(r)=\e \in (0,1]$ for any $r \geq 0.$
\item The most physically relevant variable restitution coefficient is the one corresponding to the so-called \textit{\textbf{viscoelastic hard-spheres}} \cite{BrPo}. For such a model, the properties of the restitution coefficient have been derived in \cite{BrPo,PoSc} and it can be shown that $e(z)$ is   defined implicitly by the following
\begin{equation}\label{visco}e(r) + a r^{1/5} e(r )^{3/5}=1 \qquad \qquad \forall r \geq 0\end{equation}
where $a>0$ is a suitable positive constant depending on the material viscosity.
\end{enumerate}
\noindent
With the above notations, the Boltzmann collision operator is given, in weak form, by the following equation
\begin{multline}\label{weakn}
\IR \Q_e(f,f)\psi(v)\d v\\=\dfrac{1}{2\pi}\int_{\R^3 \times \R^3 \times \S} |u \cdot \n|f(v)f(\vb)\big(\psi(v')+\psi(\vb')-\psi(v)-\psi(\vb)\big)\d v \d\vb \d\n
\end{multline}
for any smooth test-function $\psi(v)$. The strong form of $\Q_e$ can be recovered easily (see \cite{AlonsoIumj}). Notice that an alternative parametrization of the post-collision velocities \eqref{transfpre} would lead to a slighty different weak formulation of the collision  operator (this alternative formulation  was preferred in \cite{AloLo1}).

\medskip
\noindent
As explained in \cite{AloLo1}, in absence of any heating source, the granular temperature
$$\E(t)=\IR f(t,v)|v|^2\d v, \qquad \quad t \geq 0$$
is continuously decreasing and tending to zero as time goes to infinity, expressing the \textit{cooling of the granular gases}. The precise cooling rate of the temperature is the main concern of this note.  It was proven in \cite{AloLo1}, and predicted by the physics literature long ago, that the cooling rate is strongly depending on the choice of the restitution coefficient.  Note that using the weak form \eqref{weakn} with the test function $\psi(v)=|v|^2$, the evolution of $\E(t)$ is governed by the following relation
\begin{equation}\label{temperature}\dfrac{\d }{\d t}\E(t)=- \IRR f(t,v)f(t,\vb)\mathbf{\Psi}_e(|u|^2)\d v\d\vb,
\end{equation}
where the dissipation energy potential associated to $e(\cdot)$ is given by
\begin{equation}\label{Psie}
\mathbf{\Psi}_e(r)=\frac{r^{3/2}}{2} \int_0^{1} \left(1-e^2(\sqrt{r}z) \right) z^3\d z \qquad \qquad r >0.\end{equation}
We refer to \cite{AloLo1} for technical details. In the \textit{op. cit.} we introduced the following general assumptions:
 \begin{hyp}\label{HYP1} The  restitution coefficient $e(\cdot)$ is such that the following hold:
 \begin{enumerate} [(1)\:]
\item The mapping  $r \in \mathbb{R}_+ \mapsto e(r) \in (0,1]$ is absolutely continuous.
\item The mapping $r\in\mathbb{R}^{+} \mapsto \vartheta(r):=r\;e(r)$ is strictly increasing.
\item $\limsup_{r \to \infty} e(r)=e_0 < 1.$
\item The function $x>0 \longmapsto \mathbf{\Psi}_e(x)$ defined in \eqref{Psie} is strictly increasing and convex over $(0,+\infty)$.
\end{enumerate}
\end{hyp}
\noindent
These assumptions are fulfilled by the two examples described above and, more generally, they hold whenever the restitution coefficient $r \mapsto e(r)$ is an absolutely continuous and non-increasing mapping (see \cite[Appendix]{AloLo1}). Notice that the first two assumptions are exactly those needed in order to compute the Jacobian of the transformation \eqref{transfpre}. The last two are needed in order to get the following proposition which is based on Jensen's inequality.
\begin{propo}\label{prop:cool} Let $f_0$ be a nonnegative velocity distribution satisfying \eqref{initial} and let $f(t,v)$ be the associated solution to the Cauchy problem \eqref{cauch} where the variable restitution coefficient satisfies Assumptions \ref{HYP1}. Then,
$$\dfrac{\d }{\d t}\E(t) \leq -\mathbf{\Psi}_e(\E(t)) \qquad \qquad \forall t \geq 0$$
and, as a consequence, $\lim_{t \to \infty}\E(t)=0.$ Moreover, if one assumes that there exist $\alpha >0$ and $\gamma \geq 0$ such that
\begin{equation}\label{smallez}e(z) \simeq 1-\alpha\, z^\gamma \quad \text{ for } \quad z \simeq 0\end{equation}
then there exist  $C >0$ and $t_0 >0$ such that
$$\E(t) \leq C\left(1+t\right)^{-\frac{2}{1+\gamma}} \qquad \qquad \forall t \geq t_0.$$
\end{propo}
\begin{nb} The well-posedness of the Cauchy problem \eqref{cauch} has been proved in \cite{MMR}.
\end{nb}
\begin{nb} Notice that the above assumption \eqref{smallez} is equivalent to assume that
\begin{equation}\label{ell}\ell_\gamma(e)=\sup_{r >0} \dfrac{1-e(r)}{r^\gamma} < \infty \end{equation}
since, for large value of $r >0$, $(1-e(r))/r^\gamma$ is clearly finite for any $\gamma >0$.  In particular, for a constant restitution coefficient $e(r)=e_0$, one has $\ell_{0}(e) = 1-e_0$ which means that \eqref{ell} holds.  For the model of viscoelastic hard-spheres given by \eqref{visco}, the restitution coefficient $e(\cdot)$ is such that $\ell_{1/5}(e)= a$.  Furthermore, if $\ell_\gamma(e) < \infty$ for some $\gamma >0$, then $\ell_\alpha(e)=\infty$ for any $\alpha \neq \gamma$. Indeed, the parameter $\gamma$ is exactly the one that prescribes the behavior of $e(r)$ for \textit{\textbf{small values}} of $r$.\end{nb}
\noindent
Proposition \ref{prop:cool} illustrates the fact that the decay of the temperature is governed by the behavior of the restitution coefficient $e(r)$ for small value of $r$. Now, in order to match the precise cooling rate of the temperature and prove the so-called generalized Haff's law, one needs to prove that, under Assumptions \ref{HYP1} and \eqref{ell}, there exists $c>0$ such that
\begin{equation}\label{converseE}
\E(t) \geq c(1+t)^{-\frac{2}{1+\gamma}}\qquad \forall t \geq 0.
\end{equation}
This was precisely the main objective in \cite{AloLo1} and, as far as the cooling rate is concerned, the main result of the \textit{op. cit.} can be formulate as
\begin{theo}\label{haff1} For any initial distribution velocity $f_0 \geq 0$ satisfying the  conditions given by \eqref{initial} with $f_0 \in L^{p}(\R^3)$ for some $1 < p < \infty$, the solution $f(t,v)$ to the associated Boltzmann equation \eqref{cauch} satisfies the generalized Haff's law for variable restitution coefficient $e(\cdot)$ fulfilling Assumptions \ref{HYP1} and  \eqref{ell}:
\begin{equation}\label{Haff's}
c (1+t)^{-\frac{2}{1+\gamma}}  \leq \E(t) \leq  C (1+t)^{-\frac{2}{1+\gamma}}, \qquad t \geq 0
\end{equation}
where $c,C$ are positive constants.
\end{theo}
\begin{nb} An additional assumption was required in theorem \ref{haff1} on the restitution coefficient $e(\cdot)$, see \cite[Assumption 4.10]{AloLo1}.  We do not insist on this point since we believe such assumption is only of technical nature and likely unnecessary.
\end{nb}
\noindent
For constant restitution coefficient, Haff's law has been proved in \cite{MiMo}. This approach was generalized in \cite{AloLo1} leading to Theorem \ref{haff1}.  The proof is based on the following steps:
\begin{enumerate}[(1)\:]
 \item The study of the moments of solutions to the Boltzmann equation using a generalization of the Povzner's lemma developed in \cite{BoGaPa}.
 \item Precise $L^p$ estimates, in the same spirit of \cite{MiMo}, of the solution to the Boltzmann equation for $1<p<\infty$.
 \item A study of the problem \eqref{cauch} in self-similar variable (that is, for suitable \textit{rescaled solutions}).
\end{enumerate}
Because of the method of proof, step $(2)$ hereabove, the above result requires strong integrability assumption on the initial density $f_0$ which has to belong to some $L^{p}$ space with $p  >1.$ The main purpose of this paper is to remove the unphysical assumption
\begin{equation}\label{f0p}
f_0 \in L^p \qquad \text{ for some } p >1,
\end{equation}
and prove that the generalized Haff's law still holds under less restrictive assumptions.
%%%%%%%%%%%%%%%%%%%%%%%%
%%%%%%%%%%%%%%%%%%%%%%%%
\subsection{Main results} We present two independent treatments of the above problem:
\begin{enumerate}[(1)\;]
\item We prove that Haff's law still holds if we replace \eqref{f0p} by the less restrictive constraint
$$\H(f_0)=\IR f_0(v)\log f_0(v)\d v < \infty$$
and $e(\cdot)$ satisfying \eqref{ell} for some $\gamma >0$.
\item Using only finiteness of mass and energy on the initial datum, we prove that Haff's law \eqref{Haff's} holds true in some weakly inelastic regime defined in the sequel.  We notice that for constant restitution coefficient, the proof of Haff's law given in \cite{MiMo} required the assumption $f_0 \in L^p$ with $p >1$. However, it was observed in \cite{AloLo1} that, for this case, Haff's law holds assuming only that the restitution coefficient is sufficiently close to one.  We give a complete proof of this fact in the sequel.
\end{enumerate}
More precisely, the  main results of the present paper can be stated in the following theorems.
\begin{theo}
\label{main0} Let $e(\cdot)$ be a non-constant restitution coefficient that satisfies Assumptions \ref{HYP2} below.  Furthermore, assume that the initial distribution $f_0$ satisfies \eqref{initial} together with $\H(f_0) < \infty$, and let $f(t,v)$ be the unique solution to \eqref{cauch}. Then, the generalized Haff's law \eqref{Haff's} holds true.
\end{theo}
\noindent
For constant restitution coefficient, our result is weaker, however, we give a qualitative version of Haff's law in this case indicating an algebraic rate of decrease for the temperature, see Theorem \ref{entrHaff1}.

\medskip

\noindent
Referring to the point $(2)$ we state two different results, distinguishing the constant and non-constant cases.
\begin{theo}\label{mainc} Let $f_0$ be a nonnegative velocity distribution satisfying \eqref{initial} and let $f(t,v)$ be the associated solution to \eqref{cauch}. Assume that the  restitution coefficient $e(\cdot)$ is constant $e(r)=\e \in (0,1)$ and such that
\begin{equation*}\label{small*}\frac{3(1-\e^2)}{8} < 1-\kappa_{3/2}\end{equation*}
where the expression of $\kappa_{3/2}$ is given in Prop. \ref{povzner}. Then the generalized Haff's law \eqref{Haff's} holds true.
\end{theo}
\noindent
For non-constant restitution coefficient, the situation is different and the condition on the restitution coefficient will depend on the initial datum. One can formulate our result as follows (see Theorem \ref{mainnon} for a more precise statement).
\begin{theo}\label{main} Let $f_0$ be a nonnegative velocity distribution satisfying \eqref{initial} and let $f(t,v)$ be the associated solution to the Cauchy problem \eqref{cauch}. For any $\gamma>0$, there exists some explicit $\ell_0:=\ell_{0}(f_0,\gamma) >0$ such that, if the  restitution coefficient $e(\cdot)$
satisfies Assumptions \ref{HYP1} and \eqref{ell} with $\ell_\gamma(e) < \ell_0$, then the generalized Haff's law \eqref{Haff's} hold true.
\end{theo}
\noindent
\noindent
The proof  of Theorem  \ref{main0} is much simpler than the proof of \cite{AloLo1} under the assumption \eqref{f0p} on the initial datum. In particular, it does not requires the introduction of self-similar variables. It is based essentially on the fact that entropy of the solution $f(t,v)$ to \eqref{cauch} grows at most logarithmically, namely, there exists $K_0 >0$ such that
$$\H(f(t)) \leq K_0 \log (1+t) \qquad \forall t \geq 0.$$
Then, using some estimates which allow to relate the energy $\E(t)$ to the entropy, we can deduce from such logarithmic growth that the decreasing of the energy $\E(t)$ is \textit{at most} algebraic, that is, there exists some finite $\lambda >0$ such that $\inf_{t\geq0}(1+t)^\lambda \E(t)>0$. It is known from  \cite{AloLo1} (see also Proposition \ref{lambda}) that for non-constant restitution coefficient, this is enough to conclude the Haff's law \eqref{Haff's}.  The proof of Theorem \ref{main0} is given in Section 3 (see Theorem \ref{entrHaff}) while several inequalities relating energy and entropy are given in the Appendix.\\

\noindent
Concerning Theorems \ref{mainc} and \ref{main}, their proofs are surprisingly   simple and rely only on a careful study of the various moments
of the solution to the Cauchy problem \eqref{cauch}. They will be the object of Section 4.
%%%%%%%%%%%%%%%%%%%%%%%%
%%%%%%%%%%%%%%%%%%%%%%%%
\section{Some known results}
We briefly recall some known estimates on the moments of the solution to the Cauchy problem. In this section, we will assume that the restitution coefficient $e(\cdot)$ satisfies Assumptions \ref{HYP1} and that the initial datum $f_0$ satisfies \eqref{initial}. We denote then by  $f(t,v)$ the associated solution to the Cauchy problem \eqref{cauch}. For any $t \geq 0$ and any $p \geq 1$ we define
\begin{equation}\label{defmp}m_p(t):=\IR f(t,v)|v|^{2p}\d v\end{equation}
with the convention of notation $\E(t)=m_1(t)$. Then, one has the following proposition, see \cite{AloLo1}.

\begin{propo}\label{povzner} For any real $p \geq 1$, one has
\begin{equation}\label{QepSp}
\dfrac{\d}{\d t}m_p(t)=\IR \Q_{e}(f,f)(t,v)|v|^{2p}\d v\leq-(1-\kappa_{p})m_{p+1/2}(t)+\kappa_{p}\;S_{p}(t),
\end{equation}
where,
\begin{equation*}
S_{p}(t)=\sum^{[\frac{p+1}{2}]}_{k=1}\left(
\begin{array}{c}
p\\k
\end{array}
\right)\left(m_{k+1/2}(t)\;m_{p-k}(t)+m_{k}(t)\;m_{p-k+1/2}(t)\right),
\end{equation*}
$[\frac{p+1}{2}]$ denoting the integer part of $\frac{p+1}{2}$ and
$$\kappa_p=\sup_{\widehat{U} \in \mathbb{S}^2}\int_{\widehat{U}\cdot \sigma \geq 0} \left(\dfrac{3+\widehat{U}\cdot \sigma}{4}\right)^p  + \left(\dfrac{1-\widehat{U}\cdot\sigma}{4}\right)^p\dfrac{\d\sigma}{2\pi}=\int_0^1 \left(\frac{3+t}{4}\right)^p + \left(\frac{1-t}{4}\right)^p \d t >0$$ is an explicit constant such that $\kappa_p < 1$ for any $p > 1.$
\end{propo}
\noindent
A simple consequence of the above is the following, \cite[Corollary 3.6]{AloLo1}:  For any $p \geq 1$, there exists some constant $K_p >0$ such that
$$m_p(0) < \infty \implies m_p(t) \leq K_p(1+t)^{-\frac{2p}{1+\gamma}} \qquad \forall t \geq 0.$$
Furthermore, since we are dealing with hard spheres, the phenomenon of appearance of moments occurs in the same way as in the classical elastic Boltzmann \cite{desvillettes, We99}.  Thus, as soon as $\E(0) < \infty$, the higher moments satisfy $\sup_{t\geq t_0}m_p(t) < \infty$ for any $t_0 >0$.  In particular, one can rephrase \cite[Corollary 3.6]{AloLo1}.
\begin{propo}\label{moments}
Let $f_0$ be a nonnegative velocity distribution satisfying \eqref{initial} and let $f(t,v)$ be the associated solution to the Cauchy problem \eqref{cauch} where the variable restitution coefficient satisfies Assumptions \ref{HYP1} and \eqref{smallez}. For any $t_0 >0$ and any $p \geq 0$,  there exists $K_p >0$ such that
\begin{equation}\label{kp}
m_p(t) \leq K_p \left(1+t\right)^{-\frac{2p}{1+\gamma}} \qquad \forall t \geq t_0.
\end{equation}
\end{propo}
\noindent
Observe that in order to prove that the second part of Haff's law \eqref{converseE}, it is enough to control $m_{\frac{3+\gamma}{2}}(t)$ in terms of $\E(t)^{\frac{3+\gamma}{2}}$. Indeed,  recall that
$$-\dfrac{\d}{\d t}\E(t)=\IRR f(t,v)f(t,\vb)\mathbf{\Psi}_e(|u|^2)\d v \d \vb \qquad \forall t \geq 0.$$
Since $\ell_\gamma(e) < \infty$, one has
$$1-e(r) \leq \ell_\gamma(e)r^\gamma \qquad \forall r >0.$$
Plugging this estimate in the definition of $\mathbf{\Psi}_e(r^2)$ and using the fact that $1-e^2(r) \leq 2(1-e(r))$ for any $r >0$, we get that
$$\mathbf{\Psi}_e(|u|^2) \leq \ell_\gamma(e)|u|^{3+\gamma}\int_0^1 z^{3+\gamma}\d z=\dfrac{\ell_\gamma(e)}{4+\gamma} |u|^{3+\gamma} \qquad \forall u\in\mathbb{R}^{3}.$$
Since $|u|^{3+\gamma} \leq 2^{2+\gamma}\left(|v|^{3+\gamma}+|\vb|^{3+\gamma}\right)$, one gets
\begin{equation}\begin{split}\label{Egamma}
- \dfrac{\d}{\d t}\E(t)  &\leq \dfrac{2^{2+\gamma}\ell_\gamma(e)}{4+\gamma} \IRR f(t,v)f(t,\vb)\left(|v|^{3+\gamma}+|\vb|^{3+\gamma}\right)\d v\d\vb\\
&=\dfrac{2^{3+\gamma}\ell_\gamma(e)}{4+\gamma} m_{\frac{3+\gamma}{2}}(t).\end{split}
\end{equation}
Therefore, if there exists some constant $K>0$ such that
\begin{equation}\label{controlg} m_{\frac{3+\gamma}{2}}(t) \leq K \E(t)^{\frac{3+\gamma}{2}}(t) \qquad \forall t \geq t_0>0,\end{equation}
setting $C_\gamma=\frac{2^{3+\gamma}K}{4+\gamma}\ell_\gamma(e)$, we obtain from \eqref{Egamma} that
$$- \dfrac{\d}{\d t}\E(t)  \leq C_\gamma\,\E(t)^{\frac{3+\gamma}{2}} \qquad \forall t \geq t_0.$$
A simple integration of this inequality yields implies,
$$
\E(t)\geq \frac{\E(t_0)}{\left(1+\frac{1+\gamma}{2}\E(t_0)^{\frac{1+\gamma}{2}}C_{\gamma}\;(t-t_0)\right)^{\frac{2}{1+\gamma}}}\quad\forall t\geq t_0.
$$
which implies \eqref{converseE}.  An additional simplification in the arguments comes with the following proposition which has already been used implicitly in \cite{AloLo1}. We give a complete proof of it for the sake of clarity.
\begin{propo}\label{propstrat}  Assume that the restitution coefficient $e(\cdot)$ satisfy  Assumptions \ref{HYP1} and is such that $\ell_\gamma(e) < \infty$ for some $\gamma \geq 0$. Assume that there exists some constant $C>0$ such that
\begin{equation}\label{controlm}
m_{\frac{3}{2}}(t) \leq C \E(t)^{\frac{3}{2}}(t) \qquad \forall t \geq t_0>0.
\end{equation}
Then, for any $p \geq 3/2$, there exists a constant $K_p >0$ such that
\begin{equation}\label{kp} m_p(t) \leq K_p \E(t)^p \qquad \forall t \geq t_0.
\end{equation}
In particular, the generalized Haff's law \eqref{Haff's} holds.
\end{propo}
\begin{proof} Let $t_0 > 0$ be fixed.  First observe that using classical interpolation, it suffices to prove the result for any $p \geq 3/2$ such that $2p\in \mathbb{N}$.  Argue by induction assuming that for any integer $j$ such that $2j\in\mathbb{N}$, and $1 \leq j \leq p-1/2$ there exists $K_j >0$  such that $m_j(t) \leq K_j \E(t)^j$ for $t\geq t_0$.\\

\noindent
Recall that, according to Proposition \ref{povzner}
$$\frac{\d }{\d t}m_p(t) \leq-(1-\kappa_{p})m_{p+1/2}(t)+\kappa_{p}\;S_{p}(t),$$
where
\begin{equation*}
S_{p}(t)=\sum^{[\frac{p+1}{2}]}_{k=1}\left(
\begin{array}{c}
p\\k
\end{array}
\right)\left(m_{k+1/2}(t)\;m_{p-k}(t)+m_{k}(t)\;m_{p-k+1/2}(t)\right).
\end{equation*}
The crucial point is that, for $p \geq 2$, the above expression $S_p(t)$ involves moments of order less than $p-1/2$ except for $p=3/2$ which explains its peculiar role.  The induction hypothesis implies therefore that there exists a constant $C_p>0$ such that
$$
S_p(t) \leq C_p\, \E(t)^{p+1/2} \qquad \forall t \geq t_0,
$$
where $C_p$ can be taken as $$
C_p=\sum^{[\frac{p+1}{2}]}_{k=1}\left(
\begin{array}{c}
p\\k
\end{array}
\right)\left(K_{k+1/2}\;K_{p-k}+K_{k}\;K_{p-k+1/2}\right).
$$
Furthermore, according to Jensen's inequality $m_{p+1/2}(t)\geq m_{p}^{1+1/2p}(t)$, therefore we obtain
\begin{equation}\label{difmp}
\dfrac{\d }{\d t}m_p(t) \leq -(1-\kappa_{p})m_{p}^{1+1/2p}(t)+ \kappa_p\;C_p \,\E(t)^{p+1/2} \qquad \forall t \geq t_0.
\end{equation}
Additionally, according to \eqref{Phie} and since $e(r) \leq 1$ for any $r \geq 0$, one has clearly $\mathbf{\Psi}_e(|u|^2) \leq \frac{|u|^3}{8}$ for any $u \in \R^3.$ Thus, using \eqref{temperature},
\begin{align}\label{dEmm32}
-\dfrac{\d}{\d t}\E(t) &\leq \dfrac{1}{8} \IRR |u|^3 f(t,v)f(t,\vb)\d v\d\vb \nonumber\\
&\leq \IR f(t,v)|v|^3\d v=m_{3/2}(t) \qquad \forall t \geq t_0.
\end{align}
Let $K>$ be conveniently chosen later and define $U_{p}(t):=m_{p}(t)- K\E(t)^{p}$.  Then, combining \eqref{difmp} and \eqref{dEmm32},
\begin{equation*}
\begin{split}
\dfrac{\d}{\d t}U_p(t)&=\dfrac{\d}{\d t}m_p(t)-pK\E(t)^{p-1}\dfrac{\d}{\d t}\E(t)\\
&\leq -(1-\kappa_{p})m_{p}^{1+1/2p}(t)+ \kappa_p\;C_p \,\E(t)^{p+1/2} +pKm_{3/2}(t)\E(t)^{p-1} \qquad \forall t \geq t_0.
\end{split}
\end{equation*}
Therefore, using \eqref{controlm} we obtain
$$\dfrac{\d}{\d t}U_p(t) \leq -(1-\kappa_{p})m_{p}^{1+1/2p}(t)+ \kappa_p\;C_p \,\E(t)^{p+1/2} +pK\E(t)^{p+1/2} \qquad \forall t\geq t_0.$$
This is enough to prove \eqref{kp} for $K=K_p$ large enough. Indeed, pick  $K$ so that $m_{p}(t_0)<K \E(t_0)^{p}$. Then, by time-continuity in the moments, the estimate \eqref{kp} follows at least for some finite subsequent time. Assume that there exists a time $t_{\star} > t_0$ such that $m_{p}(t_{\star})=K \E(t_{\star})^{p}$, then the above inequality  implies
$$
\dfrac{\d U_{p}}{\d t}(t_\star)\leq \left(-(1-\kappa_{p})K^{1+1/2p}+\kappa_p\,C_p + pK\right)\E(t_\star)^{p+1/2} <0
$$
whenever $K$ is large enough. This proves that \eqref{kp} holds for any $p \geq 3/2$.
\end{proof}
 \noindent
For \textit{non-constant restitution coefficient}, an interesting result holds: in order to prove the lower bound \eqref{converseE}, it is enough to prove that the cooling of $\E(t)$ is \textit{at most algebraic} with arbitrary rate.  More precisely, we have  \cite[Theorem 3.7]{AloLo1} the following proposition which follows from Proposition \ref{propstrat}.
\begin{propo}\label{lambda}
Assume that the restitution coefficient $e(\cdot)$ satisfy  Assumptions \ref{HYP1} and \eqref{smallez} with $\gamma >0$. If there exist $C_0 >0$ and $\lambda >0$ such that
\begin{equation*}\label{lam}
\E(t) \geq  C_0\,(1+t)^{-\lambda} \qquad \forall t >0,
\end{equation*}
then there exists $C >0$ such that $\E(t) \geq  C\,(1+t)^{-\tfrac{2}{1+\gamma}}$ for any $t \geq 0$, i.e. the generalized Haff's law \eqref{Haff's} holds true.
\end{propo}

%%%%%%%%%%%%%%%%%%
%%%%%%%%%%%%%%%%%%
\section{Entropy-based proof of Haff's law}
The aim of this section is to prove Theorem \ref{main0}.  We begin computing the entropy production associated to the Boltzmann equation for granular gases \eqref{cauch}.
%%%%%%%%%%%%%%%%%%
%%%%%%%%%%%%%%%%%%
\subsection{Entropy production functional}
For any nonnegative $f$, one can use the weak form \eqref{weakn} with the test function $\psi(v)=\log f(v)$ to compute the production of entropy
$$\mathcal{S}_e(f):=\IR \Q_e(f,f)\log f \d v.$$
More precisely,
\begin{equation*}
\begin{split}
\mathcal{S}_{e}(f)& =\dfrac{1}{2\pi} \int_{\R^6 \times \S} |u \cdot \n|f(v)f(\vb)\log \left(\frac{f(v')f(\vb')}{f(v)f(\vb)}\right)\d v \d\vb \d\n\\
&=\dfrac{1}{2\pi} \int_{\R^6 \times \S} |u \cdot \n|f(v)f(\vb)\left(\log \left(\frac{f(v')f(\vb')}{f(v)f(\vb)}\right)-\frac{f(v')f(\vb')}{f(v)f(\vb)}+1\right)\d v \d\vb \d\n\\
&\phantom{++++}+\dfrac{1}{2\pi} \int_{\R^6 \times \S} |u \cdot \n|\left(f(v')f(\vb')-{f(v)f(\vb)}\right)\d v\d\vb\d\n.
\end{split}
\end{equation*}
Define,
\begin{equation}\label{Def}
\mathcal{D}_{e}(f):=-\dfrac{1}{2\pi} \int_{\R^6 \times \S} |u \cdot \n|f(v)f(\vb)\left(\log \left(\frac{f(v')f(\vb')}{f(v)f(\vb)}\right)-\frac{f(v')f(\vb')}{f(v)f(\vb)}+1\right)\d v \d\vb \d\n
\end{equation}
which is a non-negative quantity since $\log x \leq x-1$ for any $x >0.$ Notice that, if $e=1$, then $\mathcal{D}_{e}$ is the classical entropy production functional and $\mathcal{S}_{e}(f)=-\mathcal{D}_e(f) \leq 0$, which means that the entropy production $\mathcal{S}_e$ is non-positive.\\

\noindent
For inelastic collisions, the entropy production functional is more intricate,
$$\mathcal{S}_{e}(f)=-\mathcal{D}_e(f)+ \dfrac{1}{2\pi} \int_{\R^6 \times \S} |u \cdot \n|\left(f(v')f(\vb')-{f(v)f(\vb)}\right)\d v\d\vb\d\n$$
which means that the entropy production splits into a dissipative part $(-\mathcal{D}_e)$ and a non-negative part. Let us compute more precisely this last term. Since $\vartheta_e(\cdot)$ is strictly increasing, it is bijective.  Moreover, $|u'\cdot \n|=\vartheta_e(|u\cdot \n|))$, thus, one can write $|u\cdot\n|=\vartheta_e^{-1}(|u'\cdot\n|)$. Then, using the change of variables $(v',\vb') \to (v,\vb)$ we obtain
\begin{equation*}\begin{split}
\int_{\R^6 \times \S} |u \cdot \n| f(v')f(\vb')\d v\d\vb\d\n&=\int_{\R^6 \times \S}\vartheta_e^{-1}(|u'\cdot\n|) f(v')f(\vb')\d v\d\vb\d\n\\
&=\int_{\R^6 \times \S}\vartheta_e^{-1}(|u \cdot\n|) f(v )f(\vb )\dfrac{\d v\d\vb\d\n}{J_e(\vartheta_e^{-1}(|u\cdot \n|))}.
\end{split}\end{equation*}
We used that $$\d v'\d\vb'=J_e(|u\cdot\n|)\d v\d\vb=J_e(\vartheta_e^{-1}(|u'\cdot\n|)\d v\d\vb.$$
It is easy to see that $\vartheta_e^{-1}(|u\cdot\n|)=\frac{|u\cdot\n|}{e(\vartheta_e^{-1}(|u\cdot\n|)}$.  Then, we deduce that
\begin{multline*}
\dfrac{1}{2\pi}\int_{\R^6 \times \S} |u \cdot \n|\left(f(v')f(\vb')-{f(v)f(\vb)}\right)\d v\d\vb\d\n\\
=\dfrac{1}{2\pi}\int_{\R^6 \times \S} |u \cdot \n|{f(v)f(\vb)}\left(\dfrac{1}{e (\vartheta_e^{-1}(|u\cdot\n|)) J_e(\vartheta_e^{-1}(|u\cdot\n|))}-1\right)\d v\d\vb\d\n\\
=\IRR |u|f(v)f(\vb)\mathbf{\Phi}_e(|u|)\d v\d\vb.
\end{multline*}
For any fixed $v,\vb$, we have defined
$$\mathbf{\Phi}_e(|u|):=\dfrac{1}{2\pi}\int_{\S} |\widehat{u} \cdot \n|\left(\dfrac{1}{e (\vartheta_e^{-1}(|u\cdot\n|)) J_e(\vartheta_e^{-1}(|u\cdot\n|))}-1\right) \d\n.$$
After some minor computations,
$$\mathbf{\Phi}_e(|u|)=\frac{2}{|u|^2}\int_0^{|u|} \left(\dfrac{1}{e (\vartheta_e^{-1}(z))\, J_e(\vartheta_e^{-1}(z))}-1\right)z\d z.$$
Setting then $r=\vartheta_e^{-1}(z)$ and recalling that $J_e(y)=\vartheta'_e(y)$, we easily get that
\begin{equation*}
\mathbf{\Phi}_e(|u|)=\dfrac{2}{|u|^2}\int_0^{\vartheta_e^{-1}(|u|)} \left(r-\vartheta_e(r)\,\vartheta_e'(r)\right)\d r
\end{equation*}
where we also used that $\vartheta_e^{-1}(0)=0.$  We just proved the following proposition.
\begin{propo} Assume that the restitution coefficient $e(\cdot)$ satisfies Assumption \ref{HYP1}, items (1) and (2). Then, for any non-negative distribution function $f(v)$
\begin{equation}\label{dissip}
\mathcal{S}_{e}(f)=\IR \Q_e(f,f)\log f \d v=-\mathcal{D}_e(f) + \IRR |u|f(v)f(\vb)\mathbf{\Phi}_e(|u|)\d v\d\vb
\end{equation}
where $\mathcal{D}_e(f) \geq 0$ is given by \eqref{Def} while $\mathbf{\Phi}_e(\cdot)$ is defined by
\begin{equation}\label{Phie}
\mathbf{\Phi}_e(\varrho)=\dfrac{2}{\varrho^2}\int_0^{\vartheta_e^{-1}(\varrho)} \left(r-\vartheta_e(r)\,\vartheta_e'(r)\right)\d r, \qquad \forall \varrho >0.
\end{equation}
\end{propo}
\noindent
Additional qualitative properties of $\mathbf{\Phi}_e$ are given in the following lemma.
\begin{lemme}\label{lemPhie} Assume that the restitution coefficient $e(\cdot)$ satisfies Assumption \ref{HYP1}. Assume moreover that \eqref{smallez} is satisfied for some $\alpha >0$ and $\gamma > 0$ and that there  exist two positive constants $C >0$ and $m \geq 1$ such that
\begin{equation}\label{large}\vartheta_e^{-1}(y) \leq C y^m \qquad \text{ for large } y.\end{equation}
Then, $\mathbf{\Phi}_e(|u|) \leq C |u|^{2(m-1)}$ for large $|u|$, and $\mathbf{\Phi}_e(|u|) \simeq  2\alpha |u|^\gamma$  for small $|u| \simeq 0$.
\end{lemme}
\begin{proof} Since $\vartheta_e(\cdot)$ is assumed to be increasing, one clearly has $r-\vartheta_e(r)\vartheta_e'(r) \leq r$ for any $r \geq 0$. Therefore,
\begin{equation*}\label{boundPhie}\mathbf{\Phi}_e(|u|) \leq \left(\dfrac{\vartheta^{-1}_e(|u|)}{|u|}\right)^2 \qquad \forall u \in \R^3\end{equation*}
and the first part of the Lemma follows from \eqref{large}. Moreover, if $e(r) \simeq 1-\alpha r^\gamma$ for $r \simeq 0$ and $\gamma >0$,  then $r-\vartheta_e(r)\,\vartheta_e'(r) \simeq \alpha(2+\gamma)r^{\gamma+1}$ for $r \simeq 0$.  Since $\vartheta_e^{-1}(r) \simeq r$ for small $r$, one gets easily the second part of the result.
\end{proof}
\noindent
The case of a constant restitution coefficient is included in the previous lemma, however, in this case $\mathbf{\Phi}_e$ is explicit, we refer to \cite{MiMo, GaPaVi} for previous uses of the entropy production functional in the constant case.
 \begin{exa} [\textbf{Constant restitution coefficient}]\label{constante}  If $e(z)=\e \in (0,1]$ for any $z \geq 0$, then $J_e=\e$ and
$$\mathbf{\Phi}_e(\varrho)=\dfrac{2(1-\e^2)}{\varrho^2}\int_0^{\varrho/\e}r\d r= \dfrac{1-\e^2}{\e^2}.$$
\end{exa}

\begin{exa} If $e(\cdot)$ is the restitution coefficient for visco-elastic hard-spheres, there exists $a >0$ such that
$$e(r)+ar^{\tfrac{1}{5}}\,e^{\tfrac{3}{5}}(r)=1 \qquad \forall r >0.$$
One checks without difficulty that $e(r) \simeq a^{-\tfrac{5}{3}}r^{-\tfrac{1}{3}}$ as $r \to \infty.$
In particular,
$$\vartheta_e(r) \simeq a^{-\tfrac{5}{3}}r^{\tfrac{2}{3}} \qquad \text{ as } r \to \infty.$$
Consequently, there exists some positive constant $C_a >0$ such that $\vartheta_e^{-1}(y) \leq C_0 y^{\tfrac{3}{2}}$ for large $y >0.$ Therefore, the assumption \eqref{large} of the above Lemma is fulfilled with $m=3/2$, and one obtains that there exists some positive constant $C >0$ such that
$$\mathbf{\Phi}_e(|u|) \leq C |u| \qquad \text{ for large } \:  u \in \R^3.$$
Since \eqref{smallez} is known to hold with $\gamma=1/5$ and $\alpha=a$, we also have
$$\mathbf{\Phi}_e(|u|) \simeq 2a|u|^{1/5} \qquad \text{ as } \qquad |u| \simeq 0.$$
\end{exa}
%%%%%%%%%%%%%%%%%%%%%%
%%%%%%%%%%%%%%%%%%%%%%
\subsection{Evolution of the entropy and the temperature: Haff's law}

In all this section, we shall assume the following additional conditions on the restitution coefficient.
\begin{hyp}\label{HYP2}  Assume that the restitution coefficient fulfills Assumptions \ref{HYP1}. Moreover, assume that \eqref{smallez} and \eqref{large} holds, that is,
\begin{enumerate}[(1)\;]
\item There exist $\alpha >0$ and $\gamma \geq 0$ such that $1-e(r) \simeq \alpha r^\gamma$ as $r\simeq 0$.
\item There exist $m \geq 1+\gamma/2$ and $C >0$ such that $\vartheta_e^{-1}(y) \leq Cy^m$ for large $y.$
\end{enumerate}\end{hyp}
\begin{nb} Note that the assumption $m \geq 1+\gamma/2$ is no restrictive since the condition $(2)$ concerns large values of $y$.
\end{nb}
\noindent
Under this conditions, the growth of the entropy of the solution to \eqref{cauch} is at most logarithmic.
\begin{propo}\label{logarithm} Assume that the restitution coefficient $e(\cdot)$ satisfies Assumptions \ref{HYP2}.  In addition, assume that the initial distribution $f_0$ satisfies \eqref{initial} together with $\H(f_0) < \infty$ and let $f(t,v)$ be the solution to \eqref{cauch}. Then, there exists a constant $C_0 >0$ such that the entropy $\H(f(t))$ of $f(t,v)$ satisfies
$$\H(f(t)) \leq \H(f_0)  + C_0 \log (1+t) \qquad \forall t \geq 0.$$
\end{propo}
\begin{proof} From the results of previous section, the evolution of the entropy $\H(f(t))$ is governed by
\begin{equation}\label{dHH}
\dfrac{\d}{\d t}\H(f(t))= -\mathcal{D}_e(f(t)) +  \IRR |u|f(t,v)f(t,\vb)\mathbf{\Phi}_e(|u|)\d v\d\vb \qquad \forall t \geq 0.\end{equation}
Under Assumption \ref{HYP2},  Lemma \ref{lemPhie} implies that we have $\mathbf{\Phi}_e(|u|) \leq C |u|^{2(m-1)}$ for large $|u|$  while $\mathbf{\Phi}_e(|u|) \simeq  2\alpha |u|^\gamma$  for small $ |u| \simeq 0$.  In particular, there are two positive constants $A$ and $B$ such that
$$\mathbf{\Phi}_e(|u|) \leq A|u|^\gamma+B|u|^{2(m-1)} \qquad \forall u \in \R^3.$$
From \eqref{dHH} and since $-\mathcal{D}_e(f(t)) \geq 0$,
$$\dfrac{\d}{\d t}\H(f(t)) \leq A\IRR |u|^{\gamma+1} f(t,v)f(t,\vb)\d v\d\vb + B\IRR |u|^{2m-1} f(t,v)f(t,\vb)\d v\d\vb.$$
Moreover, since $|u|^{\gamma+1} \leq 2^\gamma \left(|v|^{\gamma+1}+|\vb|^{\gamma+1}\right)$ one has
$$\IRR |u|^{\gamma+1} f(t,v)f(t,\vb)\d v\d\vb \leq 2^{\gamma+1}\IRR |v|^{\gamma+1} f(t,v)\d v=2^{\gamma+1}m_{\frac{\gamma+1}{2}}(t).$$
Moreover, setting $p=m-1/2$, one notices that there is a constant $c_p$ depending only on $p$ such that
$$\IRR |u|^{2m-1} f(t,v)f(t,\vb)\d v\d\vb \leq m_p(t)$$
where the $m_p(t)$ terms are the $p^{\text{th}}$ order moments defined in \eqref{defmp}.  Using Proposition \ref{prop:cool} together with Proposition \ref{moments}, one concludes that there exist two positive constants $C_1$ and $C_2$ such that
$$\dfrac{\d}{\d t}\H(t) \leq C_1 (1+t)^{-1} + C_2  \left(1+t\right)^{-\frac{2p}{1+\gamma}} \qquad \forall t > 0.$$
Since $p=m-1/2$ with $m \geq 1+\gamma/2$ one has $\tfrac{2p}{1+\gamma} \geq 1$ and, setting $C_0=C_1+C_2$, we get
$$\dfrac{\d}{\d t}\H(t) \leq C_0 (1+t)^{-1} \qquad \forall t > 0$$
which yields the conclusion.
\end{proof}
\noindent
The above logarithmic growth is exactly what we need to prove Haff's law. Indeed, the following general result allows to control from below the temperature using the entropy. The proof of the following proposition is given in the Appendix.
\begin{propo}\label{propcontrol} Let $\mathcal{C}$ denote the class of  nonnegative velocity distributions $f=f(v)$ with unit mass, finite energy and finite entropy
$$\int_{\R^3} f(v)\d v=1 \qquad \qquad \int_{\R^3} |v|^2 f(v)\d v < \infty, \qquad \H(f)=\int_{\R^3} f(v)\log f(v)\d v < \infty.$$
Define
$$\bH(f)=\IR f(v)|\log f(v)|\d v, \qquad f \in \mathcal{C}.$$
Then, there is some constant $c >0$ such that
$$\bH(f) \leq \H(f) + c \left(\IR f(v)\,|v|^2\d v\right)^{5/3},$$
and some other constant $C >0$ such that
\begin{equation}\label{control}\IR f(v)|v|^2\d v \geq C \exp\left(-\tfrac{4}{3} \bH(f)\right) \qquad \forall f \in \mathcal{C}.
\end{equation}
\end{propo}
\noindent
Proposition \ref{propcontrol} combined with the logarithmic growth of $\H(f(t))$ prove the Haff's law for non-constant restitution coefficient.
\begin{theo}\label{entrHaff} Assume that the restitution coefficient $e(\cdot)$ satisfies Assumptions \ref{HYP2} with $\gamma >0$. Let the initial distribution $f_0$ satisfies \eqref{initial} and $\H(f_0) < \infty$, and let $f(t,v)$ be the unique solution to \eqref{cauch}. Then,  there is a constant $c >0$ such that
$$\E(t) \geq c\left(1+t\right
)^{-\frac{2}{1+\gamma}} \qquad \forall t \geq 0.$$
In particular, the generalized Haff's law \eqref{Haff's} holds true.
\end{theo}
\begin{proof} With the notations of the above Proposition \ref{propcontrol},  there is some constant $C >0$ independent of time such that
$$\bH(f(t)) \leq \H(f(t)) + C\E(t)^{5/3} \qquad \forall t \geq 0.$$
Since $\E(t)$ is bounded, one can find constants $k_0, K_0 >0$ such that
\begin{equation}\label{K0}
\bH(f(t)) \leq k_0+ K_0 \log (1+t) \qquad \forall t \geq 0.
\end{equation}
Proposition \ref{propcontrol} also implies that there exists a constant $c_1 >0$ such that
$$\E(t) \geq c_1 \exp\left(-\frac{4}{3}\bH(f(t))\right) \geq c_2 \exp\left(-\frac{4K_0}{3}\log (1+t)\right) \qquad \forall t \geq 0,$$
with $c_2=c_1 \exp(-\frac{4k_0}{3}).$ Therefore, setting $\lambda_0=\frac{2K_0}{3}>0$ we get that
\begin{equation}\label{lambda0} \E(t) \geq c_2 (1+t)^{-2\lambda_0} \qquad \forall t \geq 0.\end{equation}
Therefore, according to Proposition \ref{lambda}, the estimate \eqref{lambda0}  is enough to prove the second part of Haff's law \eqref{converseE}.
\end{proof}
\noindent
For constant restitution coefficient, our result is less precise, however, we prove an \textit{integrated version of Haff's law} in the next theorem.
\begin{theo} \label{entrHaff1} Let $\e \in (0,1)$ be a constant restitution coefficient and let the initial distribution $f_0$ satisfies \eqref{initial} with $\H(f_0) < \infty$ , and let $f(t,v)$ be the unique solution to \eqref{cauch}. Then, there are two positive constants $a,b >0$ such that
\begin{equation}\label{intEt}
\int_0^t \sqrt{\E(s)}\d s \geq \dfrac{1}{b}\log\left(1+ab\, t\right) \qquad \forall t \geq 0.\end{equation}
Consequently,
$$\sup\left\{\lambda \geq 0\,,\,\sup_{t \geq 0}(1+t)^{2\lambda}\E(t) < \infty\right\}= \inf\left\{\lambda >0\,;\,\limsup_{t \to \infty} (1+t)^{2\lambda}\,\E(t) >0\right\}=1.$$
\end{theo}
\begin{proof}   Recall that, for constant restitution coefficient $\e \in (0,1)$ the evolution of the entropy is given by
$$\dfrac{\d}{\d t}\H(f(t))=- \mathcal{D}_e(f(t)) + \dfrac{1-\e^2}{\e^2} \IRR |u|f(t,v)f(t,\vb) \d v\d\vb \qquad \forall t \geq 0$$
where we used  Eq. \eqref{dHH} and Example \ref{constante}. Arguing as in the previous proof, we see that
$$\dfrac{\d}{\d t}\H(f(t)) \leq 2\dfrac{1-\e^2}{\e^2}m_{1/2}(t) \leq 2\dfrac{1-\e^2}{\e^2}\sqrt{\E(t)} \qquad \forall t \geq 0.$$
Integrating this inequality yields
$$\H(f(t)) \leq \H(f_0) + 2\dfrac{1-\e^2}{\e^2}\int_0^t \sqrt{\E(s)}\d s \qquad \forall t \geq 0.$$
Consequently, there exists a positive constant $K_1  >0$ such that
$$\bH(f(t)) \leq K_1+ 2\dfrac{1-\e^2}{\e^2} \int_0^t \sqrt{\E(s)}\d s \qquad \forall t \geq 0.$$
Then, from Proposition \ref{propcontrol},
$$\E(t) \geq C \exp\left(-\frac{4K_1}{3} -\frac{8(1-\e^2)}{3\e^2}\int_0^t \sqrt{\E(s)}\d s \right) \qquad \forall t \geq 0,$$
i.e. there are two positive constants $a=\sqrt{C}\exp(-\tfrac{2K_1}{3})$ and $b=\tfrac{4(1-\e^2)}{3\e^2}$ such that
$$\sqrt{\E(t)} \geq a\exp\left(-b\int_0^t \sqrt{\E(s)}\d s\right) \qquad \forall t \geq 0$$
which yields \eqref{intEt}.  Now, setting
$$\mathcal{A}=\{\lambda \geq 0\,;\, \sup_{t \geq 0} (1+t)^{2\lambda} \E(t) < \infty\}$$
we know from  Proposition \ref{prop:cool}  that $1 \in \mathcal{A}$, i.e. $\mathcal{A} \neq \varnothing$.  Then, it follows from \eqref{intEt} that $\sup \mathcal{A}=1$.  Now, let us define
$$\mathcal{B}=\{\lambda \geq 0\,;\,\limsup_{t \to \infty} (1+t)^{2\lambda}\E(t) >0\}.$$
Notice that inequality \eqref{lambda0} holds for constant restitution coefficient. In particular, it proves that $\mathcal{B}\neq \varnothing$.  Notice also that if $\lambda_1 \in \mathcal{B}$, then any $\lambda_2 \geq \lambda_1$ belongs to $\mathcal{B}$.  We argue by contradiction to prove that $\inf \mathcal{B} =1$. Otherwise, from the previous observation, one would have $\inf \mathcal{B} =\lambda_{\mathcal{B}} > 1$. Pick  $\lambda\in(1,\lambda_{\mathcal{B}})$, it follows that $\lambda \notin \mathcal{B}$, that is,
$\limsup_{t \to \infty} (1+t)^{2\lambda}\E(t)=0$ and, in particular, $\lambda \in \mathcal{A}$.  This is impossible since $\sup \mathcal{A} =1$, and thus,  $\inf \mathcal{B}=1$.
\end{proof}
\begin{nb} Though less precise that the converse inequality \eqref{converseE}, the above integrated version of Haff's law asserts that $(1+t)^{-2}$ is the only possible {\textbf{ algebraic rate}} for the cooling of the temperature $\E(t)$. Notice that $\mathbf{v}_T(t)=\sqrt{\E(t)}$ is proportional to the so-called thermal velocity \cite{BrPo} and we may wonder what is the physical relevance of the above identity \eqref{intEt}.
Finally, we recall that it is expected the existence of a self-similar profile $\Phi_H(\cdot)$, an \textit{\textbf{homogeneous cooling state}}, such that
$f(t,v)= \mathbf{v}_T(t)^{-3}{\Phi}_H\big(\frac{v}{\mathbf{v}_T(t)}\big)$
is a solution to \eqref{cauch}.  The existence of homogeneous cooling state has been proven in \cite{MiMo} (with a slightly different definition where $\mathbf{v}_T(t)$ was replaced by $(1+t)^{-1}$) and  where the self-similar profile $\Phi_H$ is, by construction,  satisfying $\Phi_H \in L^p$ for some $p >1$.   We conjecture that the existence of such an homogeneous cooling state can be obtained using only entropy estimates, that is $\H(\Phi_H) < \infty$.
\end{nb}
\begin{nb} Notice that, since $\E(t) \leq C(1+t)^{-\tfrac{2}{1+\gamma}}$, it follows from \eqref{control} that there is some constant $K >0$ such that
 $$\bH(f(t)) \leq K\log(1+t) \qquad \forall t \geq 0$$
 which means that the logarithmic growth obtained in Proposition \ref{logarithm} is optimal.
\end{nb}
%%%%%%%%%%%%%%%%%%%%%
%%%%%%%%%%%%%%%%%%%%%
\section{Haff's law in the weakly inelastic regime}
This section is devoted to the proof of Theorems \ref{mainc} and \ref{main}. Let us now explain briefly the strategy of proof to get a precise version of generalized Haff's law.  Note that due to Proposition \ref{prop:cool}, one only has to prove a lower bound of the type
$$\E(t) \geq c(1+t)^{-\frac{2}{1+\gamma}}$$
for some positive constant $c >0$ independent of time.  The following approach uses only the evolution of some moments of the solution $f(t,v)$ with the particular use of Proposition \ref{propstrat}. We will distinguish between the case of a constant restitution coefficient $\gamma=0$ and the non-constant case $\gamma>0$ since the two results are different.
\subsection{The case of a constant restitution coefficient}\label{constant}
We assume here that the restitution coefficient $e(\cdot)$ is constant: $e(r)=\e$ for for any $r \in \R_+$.   In this case,
$$\mathbf{\Psi}_e(r)=\dfrac{1-\e^2}{8}r^{3/2} \qquad \forall r >0$$
and the evolution of the temperature is given by
\begin{equation*}\dfrac{\d}{\d t}\E(t)=-\dfrac{1-\e^2}{8}\int_{\R^3 \times \R^3} f(t,v)f(t,\vb)|v-\vb|^3\d v\d\vb, \qquad t \geq 0.\end{equation*}
Since $|v-\vb|^3 \leq (|v|+|\vb|)^3=|v|^3+|\vb|^3+3|v|^2|\vb|+3|v||\vb|^2,$ one deduces that
\begin{equation}\label{Em32}-\dfrac{\d }{\d t}\E(t)\leq \frac{1-\e^2}{4}m_{3/2}(t)+ 3\frac{1-\e^2}{4}\,\E(t)m_{1/2}(t) \qquad \forall t \geq 0.
\end{equation}
This observation will help us in proving the following theorem.
\begin{theo}\label{momentse} Let $f_0$ be a nonnegative velocity distribution satisfying \eqref{initial} and let $f(t,v)$ be the associated solution to \eqref{cauch}. Assume that the constant restitution coefficient $\e$ is such that
\begin{equation}\label{small}\frac{3(1-\e^2)}{8} < 1-\kappa_{3/2}.
\end{equation}
Then, for any $t_0 >0$, there is an explicit positive constant $C_0 >0$ such that
\begin{equation}\label{k32}
m_{3/2}(t) \leq C_0\, \E(t)^{3/2} \qquad \forall t \geq t_0.\end{equation}
Consequently, the second part of Haff's law holds
$$\E(t) \geq \dfrac{\E(0)}{\left(1+C_0\,\sqrt{\E(0)}(1-\e^2)t\right)^2} \qquad \forall t \geq 0.$$
\end{theo}
\begin{proof} Let $t_0 >0$ be fixed. According to Proposition \ref{povzner}
$$\dfrac{\d }{\d t}m_{3/2}(t)\leq-(1-\kappa_{3/2})m_{2}(t)+\kappa_{3/2}\;S_{3/2}(t).$$ From the expression of $S_{3/2}(t)$ one gets
\begin{equation}\label{m3/2}
\dfrac{\d }{\d t}m_{3/2}(t) \leq -(1-\kappa_{3/2})m_{2}(t)+m_{3/2}(t)m_{1/2}(t)+\E^{2}(t) \qquad \forall t \geq t_0
\end{equation}
where we used the fact that $\kappa_{3/2} < 1.$ Let $K $ be a positive number to be chosen later and define
$$
U_{3/2}(t):=m_{3/2}(t)- K\E(t)^{3/2}.
$$
From Holder's inequality,
\begin{equation}\label{m3/2E}
m_{3/2}(t)\leq \sqrt{\E(t)} \sqrt{m_{2}(t)}  \quad \text{ and } \quad m_{1/2}(t)\leq \sqrt{\E(t)}  \qquad \forall t \geq 0,
\end{equation}
so that
\begin{equation*}
\dfrac{\d U_{3/2}}{\d t}(t) \leq  -(1-\kappa_{3/2}) \dfrac{m_{3/2}^2(t)}{\E(t)}
+\sqrt{\E(t)} m_{3/2}(t)+ \E^2 (t)-\frac{3}{2}K\frac{\d \E(t)}{\d t}\E(t)^{1/2}.
\end{equation*}
Now, using \eqref{Em32}, one gets
\begin{multline}\label{p3/2}\dfrac{\d U_{3/2}}{\d t}(t)\leq -(1-\kappa_{3/2}) \dfrac{m_{3/2}^2(t)}{\E(t)}
+\sqrt{\E(t)} m_{3/2}(t)+ \E^2 (t)\\
+  \frac{3(1-\e^2)}{8}Km_{3/2}(t)\E(t)^{1/2}+\frac{9(1-\e^2)}{8}K\E(t)^{2}. \qquad \forall t >t_0.\end{multline}
This last inequality, together with the smallness assumption \eqref{small} imply the result  provided $K$ is large enough.  Indeed, choose $K$ so that $m_{3/2}(t_0)<K \E^{3/2}(t_0)$.  Then, by time-continuity of the moments, the result follows at least for some finite time.  Assume that there exists a time $t_{\star}>t_0$ such that $m_{3/2}(t_{\star})=K \E^{3/2}(t_{\star})$, then \eqref{p3/2} implies
$$
\dfrac{\d U_{3/2}}{\d t}(t_\star)\leq \left(-aK^2+1+\left(1+\frac{9(1-\e^2)}{8}\right)K\right)\E^2(t_\star)
$$
where $a:=(1-\kappa_{3/2})-\frac{3}{8}(1-\e^2) >0$ by \eqref{small}.  Choosing $K=C_0$ large enough such that $\frac{\d U_{3/2}}{\d t}(t_\star)\leq 0$ for any $t_{\star}$ the conclusion holds.
\end{proof}
\begin{nb} We just proved that, under the sole assumption \eqref{small}, the solution $f(t,v)$ to the Boltzmann equation \eqref{cauch} satisfies the generalized Haff's law
\begin{equation}\label{Haff's}
c (1+t)^{-2}  \leq \E(t) \leq  C (1+t)^{-2}, \qquad t \geq 0
\end{equation}
where $c,C$ are positive constants depending only on  $\E(0)$.
\end{nb}\begin{nb} Using the explicit expression of $\kappa_{3/2}$ given by Proposition \ref{povzner}, one sees that \eqref{small} amounts to
$\e \geq \sqrt{\tfrac{8\kappa_{3/2}}{3}-\tfrac{5}{3}} \simeq 0.809.$
\end{nb}
%%%%%%%%%%%%%%%%%%
%%%%%%%%%%%%%%%%%%
\subsection{Non-constant case $\gamma>0$}\label{sec:haff}
We consider in this section the non-constant case $\gamma>0$ with restitution coefficient $e(\cdot)$ satisfying Assumption \ref{HYP1} with
$$\ell_\gamma(e)=\sup_r \dfrac{1-e(r)}{r^\gamma} < \infty \qquad \text{ for some } \gamma  > 0.$$
The proof is more involved but still based on Proposition \ref{propstrat}. We therefore need only to estimate $m_{3/2}(t)$.  We will work with the following class of initial datum:  Let $\E_0$, $\varrho_0$  be two fixed positive constants. Define $\mathscr{F}(\E_0, \varrho_0)$ as the set of nonnegative distributions $g \in L^1_{3}$ such that
$$\IR g(v)\d v =1, \qquad \quad \IR g(v) v \d v =0, \qquad \IR |v|^2 g(v)\d v=\E_0 $$
 and
$$\IR g(v)|v|^3 \d v \leq \varrho_0 \left(\IR g(v)|v|^2 \d v\right)^{3/2}.$$
With this definitions we have the following result.
\begin{theo}\label{mainnon} Let $f_0$ be a nonnegative velocity distribution satisfying \eqref{initial} with initial energy $\E_0$.  Then, there exists  $\ell_0:=\ell_0(\E_0,\gamma) >0$ such that, if the restitution coefficient satisfies  $\ell_\gamma(e) < \ell_0$, then solution $f(t,v)$ to \eqref{cauch} fulfills the generalized Haff's law \eqref{Haff's}.
\end{theo}
\begin{proof}
Throughout the proof, we simply denote $\ell_\gamma(e)=\lambda$ since both $\gamma$ and $e(\cdot)$ are fixed.  Fix a time $t_0>0$, due to appearance of moments, there exists a $e$-independent constant $C:=C(t_0,\E_0,\gamma)>0$ such that $\sup_{t\geq t_0}m_{2\gamma}(t)\leq C$.  Note that $f(t_0,v)\in\mathscr{F}(\E_{t_0}, \varrho_{t_0})$ where $$\frac{m_{3/2}(t_0)}{m_{1}(t_0)^{3/2}}<\varrho_{t_0}<\infty.$$
Additionally, recall that
\begin{equation}\label{m3/2bis}
\dfrac{\d }{\d t}m_{3/2}(t) \leq -2\alpha m_{2}(t)+m_{3/2}(t)m_{1/2}(t)+\E^{2}(t) \qquad \forall t \geq t_0,
\end{equation}
where we have set $2\alpha=(1-\kappa_{3/2}) >0$. Now, using \eqref{Egamma}
\begin{align*}
-\E(t)^{1/2}\frac{\d}{\d t}\E(t)&\leq   \lambda\, c_\gamma  \E(t)^{1/2}m_{(3+\gamma)/2}(t)\\&\leq \lambda\, c_\gamma \left(\frac{1}{4}\E(t)^{2}+\frac{3}{4}m^{4/3}_{(3+\gamma)/2}(t)\right)\nonumber\\
&\leq \lambda\, c_\gamma\left(\frac{1}{4}\E(t)^{2}+\frac{3}{4}m^{1/3}_{2\gamma}(t)m_2(t)\right)\quad \forall t\geq t_0.
\end{align*}
For $t\geq t_0$ the quantity $m_{2\gamma}(t)^{1/3}$ is controlled uniformly by $A:=C^{1/3}$, therefore,
\begin{equation}\label{lbe2}
-\E(t)^{1/2}\frac{\d}{\d t}\E(t) \leq \lambda\, \frac{c_\gamma}{4}\left(\E(t)^{2} + 3A m_2(t)\right) \qquad \forall t \geq t_0.
\end{equation}
Set $U_{3/2}(t):=m_{3/2}(t)-K\E(t)^{3/2}$ where the constant $K >0$ will be suitable determined later on.  Use estimates \eqref{m3/2bis} and \eqref{lbe2} to get for all $t\geq t_0$
\begin{equation*}\begin{split}
\dfrac{\d }{\d t}U_{3/2}(t) &\leq -2\alpha m_{2}(t)+m_{3/2}(t)m_{1/2}(t)+\E^{2}(t) +  \frac{3c_\gamma}{8}\lambda\,K \E(t)^{2}\\
&\hspace{7cm} + \frac{9c_\gamma}{8}\lambda\,K A m_2(t)\\
&\leq  -\alpha m_{2}(t)+m_{3/2}(t)m_{1/2}(t)+ \left(1+\lambda\,K \frac{3c_\gamma}{8}\right)\E(t)^2\\
&\hspace{6cm} +\left(-\alpha + \lambda\,K \frac{9c_\gamma}{8}A \right)m_2(t).
\end{split}\end{equation*}
For the first term in the left-hand side, one uses the same estimate as in Theorem \ref{moments}
$$-\alpha m_{2}(t)+m_{3/2}(t)m_{1/2}(t)  \leq -\alpha \E(t)^{-1}m_{3/2}^2(t) +  m_{3/2}(t)\E(t)^{1/2},$$
thus,
\begin{multline}\label{tt}
\dfrac{\d }{\d t}U_{3/2}(t) \leq -\alpha \E(t)^{-1}m_{3/2}^2(t) +  m_{3/2}(t)\E(t)^{1/2} + \left(1+ 3K  c_\gamma\right)\E(t)^2\\
+\left(-\alpha + \lambda\,K \frac{9c_\gamma}{8}A \right)m_2(t) \quad\forall t\geq t_0.
\end{multline}
Without  loss of generality, we have assumed $\lambda < 8$ (any other positive number would have worked). Now, let $K_0 >0$ be the positive root of $-\alpha X^2+(1+3c_\gamma)X+1=0$. Then, if $K =\max(K_0, \varrho_{t_0})$, one gets that
$$\lambda < \ell_0 \implies m_{3/2}(t) \leq K\E^{3/2}(t)\quad \forall t\geq t_0$$
where $\ell_0=\frac{8\alpha}{9Ac_\gamma\,K}.$ Indeed, since $U_{3/2}(t_0) < 0$, by a continuity argument, then $U_{3/2}(t)$ remains nonpositive at least for some finite time.  Assume that there exists a time $t_{\star}>t_0$ such that $U_{3/2}(t_\star)=0$, that is, $m_{3/2}(t_{\star})=K \E^{3/2}(t_{\star})$. Then, from \eqref{tt}
$$\dfrac{\d}{\d t}U_{3/2}(t_\star) \leq \left(-\alpha K^2 + (1+3c_\gamma)\,K+1\right)\E(t_\star)^2
+\left(-\alpha + \lambda\,K \frac{9c_\gamma}{8}A \right)m_2(t_\star).$$
Since $K \geq K_0$, the first term on the right-hand side is negative while, by choice of $\lambda < \ell_0:=\ell_0(t_0,\E_0,\gamma)$, the last term is  negative as well.
In other words, $\dfrac{\d}{\d t}U_{3/2}(t_\star) < 0$ from which we deduce that $U_{3/2}(t)$ will remain non-positive for all $t\geq t_0$.  As explained in the paragraph precedent to Proposition \ref{propstrat}, this yields Theorem \ref{mainnon}.
\end{proof}
\begin{nb} It is important to notice that the smallness condition is depending on the initial datum $f_0$. This is the major difference with respect to the constant case where the smallness assumption \eqref{small} is universal.
\end{nb}
%%%%%%%%%%%%%%%%%%%%%%%%
%%%%%%%%%%%%%%%%%%%%%%%%
\section*{Appendix: Functional inequalities relating moments and entropy}
\setcounter{equation}{0}
\renewcommand{\theequation}{A.\arabic{equation}}
In this section, we present some functional inequalities that relate moments and entropy. We present the results in $\RN$ with $n \geq 1$ regardless we only use them in dimension $n=3$.  In this section, for any measurable subset  $E \subset \RN$, $|E|$ will stand for the Lebesgue measure of $E$.
Let $f=f(v)$ be a nonnegative distribution and denote
$$\H(f)=\int_{\RN} f(v)\log f(v)\d v, \qquad \mathbf{H}(f)=\int_{\RN} f(v)|\log f(v)|\d v,$$
while, for any  $k \geq 0$ we set
$$M_k(f) =\int_{\RN} f( v)|v|^k\d v.$$
First recall the following simple estimate which can be traced back to \cite{diperna}:
\begin{lemmeA} If  there is $k \in (0,n)$ such that $M_k(f) < \infty$ and $\H(f) < \infty$ then $\bH(f) < \infty.$ More precisely, for any $k \in (0,n)$ there exists some positive  constant $c_{n,k} >0$ such that
\begin{equation}\label{estimHfk}0 \leq \mathbf{H}(f) \leq \H(f) +c_{n,k} \,M_k(f)^{\tfrac{n}{n+k}}.\end{equation}
\end{lemmeA}
\begin{proof} Although it is almost explicitly stated in \cite{diperna}, we provide the complete proof of the above estimate. Define  $A=\{v \in \RN\,,\,f(v) < 1\}$, thus
\begin{equation*}
\begin{split}
\mathbf{H}(f)&=\int_{\RN \setminus A} f(v)\log f( v)\d v - \int_{A} f( v)\log f( v)\d v=\H(f) -2\int_{A} f(v)\log f(v)\d v\\
&=\H(f) + 2 \int_{A} f(v)\log \left(\frac{1}{f(v)}\right)\d v.
\end{split}
\end{equation*}
For any $a >0$, let $B =\{v \in \RN\,;\,f( v) \geq \exp(-a|v|^k)\}$ and $B^c=\RN \setminus B.$ If $v \in A  \cap B$ then $\log(\tfrac{1}{f(v)}) \leq a|v|^k$ and
\begin{equation*}
\begin{split}
\mathbf{H}(f)&\leq \H(f) + 2a\int_{A  \cap B} f(v)|v|^k\d v + 2 \int_{B^c \cap A} f(v)\log \left(\frac{1}{f(v)}\right)\d v\\
 &\leq \H(f) + 2aM_k(f)+ 2 \int_{B^c} f( v)\log \left(\frac{1}{f( v)}\right)\d v.\end{split}
\end{equation*}
Since $x\log(1/x) \leq M \sqrt{x}$ for any  $x \in (0,1)$, where $M= 2\exp(-1)$, we get
$$\int_{B^c} f(v)\log \left(\frac{1}{f(v)}\right)\d v \leq   M\int_{B^c} \sqrt{f(v)}\d v \leq M \int_{\RN} \exp\left(-a\frac{|v|^k}{2}\right)\d v.$$
Set
$$J_{n,k}(a)=\int_{\RN}  \exp\left(-\frac{a}{2}|v|^k\right)\d v=\dfrac{|\mathbb{S}^{n-1}|}{k}\left(\dfrac{2}{a}\right)^{\tfrac{n}{k}}\Gamma\left(\frac{n}{k}\right)$$
where $\Gamma(\cdot)$ is the Gamma function. We have $\mathbf{H}(f) \leq \H(f) + 2aM_k(f) + 2J_{n,k}(a)$ for any $a >0.$ Therefore, for any $k \in (0,n)$, there exists some positive constant  $C_{n,k} >0$ such that
$$\mathbf{H}(f) \leq \H(f)+2aM_k + C_{n,k} a^{-\tfrac{n}{k}} \qquad \forall a >0.$$
Optimizing the parameter $a >0$ yields \eqref{estimHfk} for some explicit constant $c_{n,k}.$
\end{proof}
\noindent
We now give a general estimate which allow to control $M_k(f)$ from below in terms of $\H(f)$. It is likely that such an estimate is well-known by specialists.  We include a proof below.
\begin{propoA}\label{entropy-moment} Let $f \geq 0$ be such that $\bH(f) < \infty$ with $\int_{\RN} f(v)\d v=1.$ Then, for any $k \geq 0$ and any $\varepsilon >0$, there exists $C(n,k,\varepsilon) >0$ independent of $f$ such that
\begin{equation}\label{mkfbhf}M_k(f) \geq  C(n,k,\varepsilon) \exp\left(-\frac{k}{n(1-\varepsilon)}\mathbf{H}(f)\right).\end{equation}
\end{propoA}
\begin{proof} For any $R >0$, let  $B_R$ denote the ball with center in the origin and radius $R$, and let $B_R^c$ be its complement.  Then,
\begin{multline}\label{BR}
M_k(f)=\int_{B_R} f(v)|v|^k\d v + \int_{B_R^c}f(v)|v|^k\d v \\
\geq R^k \int_{B^c_R} f(v)\d v=R^k\left(1-\int_{B_R}f(v) \d v\right).
\end{multline}
Therefore, in order to control $M_k(f)$ from below it is enough to control the mass of $f$ on a suitable ball. Let us fix $R >0$ and recall the generalized Young's inequality
$$xy \leq x\log x -x+\exp(y) \qquad \forall x >0, \qquad \forall y \in \mathbb{R}$$
or, equivalently,
\begin{equation*}\label{young}xy \leq \frac{x}{\lambda}\log \frac{x}{\lambda} -\frac{x}{\lambda}+\exp(\lambda\,y) \qquad \forall x >0, \lambda >0,  \qquad \forall y \in \mathbb{R}.\end{equation*}
Using this inequality with $x=f(v)$ and $y=1$ and integrating the inequality over $B_R$ we obtain
$$\int_{B_R} f(v)\d v \leq \dfrac{1}{\lambda}\int_{B_R} f(v)\log f(v)\d v -\dfrac{1}{\lambda} \left(\log \lambda + 1\right)\int_{B_R} f(v)\d v + \exp(\lambda) |B_R| \qquad \forall \lambda >0.$$
Consequently, since $\int_{B_R} f(v)\log f(v)\d v \leq \bH(f)$, we get
$$\dfrac{1}{\lambda}\left(\log \lambda + \lambda+  1\right)\int_{B_R} f(v)\d v \leq \dfrac{1}{\lambda} \mathbf{H}(f) + \exp(\lambda) |B_R| \qquad \forall \lambda >0,$$
that is,
$$\left(\log (\lambda \exp(\lambda))+1\right)\int_{B_R} f(v)\d v \leq  \mathbf{H}(f) + \lambda \exp(\lambda) |B_R|\qquad \forall \lambda >0.$$
Set $x=\log(\lambda \exp(\lambda))$ and assume $x >-1$. The last inequality becomes
$$\int_{B_R} f(v)\d v \leq \dfrac{\mathbf{H}(f)+\exp(x)|B_R|}{1+x} \qquad \forall x > -1.$$
We optimize the parameter  $x >-1$ by noticing that the right-hand side reaches its minimal value for $x_0$ such that
$$x_0 \exp(x_0)|B_R|=\mathbf{H}(f).$$
The mapping  $x \mapsto x\exp(x)$ is strictly increasing over $(-1,\infty)$ and we define  $W$ its inverse (Lambert function). We get then $x_0=W \left(\mathbf{H}(f)/|B_R|\right)$ and
$$\int_{B_R} f(v)\d v \leq \dfrac{\mathbf{ H}(f)+\exp(x_0)|B_R|}{1+x_0}=\exp(x_0)|B_R|=\exp\left(W \left(\frac{\mathbf{H}(f)}{|B_R|}\right)\right)|B_R|.$$
Combining this with \eqref{BR} we get
$$M_k \geq R^k \left(1-\exp\left(W \left(\frac{\mathbf{H}(f)}{|B_R|}\right)\right)|B_R|\right) \qquad \forall R >0$$
and we still have to optimize the parameter  $R$.  For simplicity, set  $Y= \mathbf{H}(f)$ and $X=|B_R|$. For any $\varepsilon \in (0,1)$,
$$1-\exp(W(Y/X))X=\varepsilon \Longleftrightarrow W(Y/X)=\log\left(\frac{1-\varepsilon}{X}\right)\Longleftrightarrow Y/X=\frac{1-\varepsilon}{X}\log\left(\frac{1-\varepsilon}{X}\right),$$
where we used that  $W^{-1}(t)=t\exp(t)$. Thus,
$$1-\exp(W(Y/X))X=\varepsilon \Longleftrightarrow \frac{Y}{1-\varepsilon}=\log\left(\frac{1-\varepsilon}{X}\right) \Longleftrightarrow X=(1-\varepsilon)\exp\left(-\frac{Y}{1-\varepsilon}\right).$$
To summarize, for any $\varepsilon >0$, if $R_0(\varepsilon) >0$ is such that $|B_{R_0}|=(1-\varepsilon)\exp\left(-\frac{\mathbf{ H}(f)}{1-\varepsilon}\right)$, then
$$M_k \geq \varepsilon R_0^k.$$
Since $|B_{R_0}|=\tfrac{|\mathbb{S}^{n-1}|}{n} R_0^n$ we get our conclusion with $C(n,k,\varepsilon)=\varepsilon\left(\frac{n(1-\varepsilon)}{|\mathbb{S}^{n-1}|}\right)^\frac{k}{n}.$
\end{proof}
\begin{nbA} In dimension $n=3$ with $k=2$ and $\varepsilon=1/2$, one sees that there is some constant $C >0$ such that
\begin{equation}\label{estimateH}\int_{\mathbb{R}^3} f(v)|v|^2 \d v \geq \frac{1}{2}\left(\frac{3}{8\pi}\right)^{2/3}\exp(-4\mathbf{H}(f)/3)\end{equation}
for any nonnegative distribution function $f(v) >0$ with $\IR f(v)\d v=1.$
\end{nbA}
%%%%%%%%%%%%%%%%%%%%%
%%%%%%%%%%%%%%%%%%%%%

\end{document}